\documentclass[11pt]{article}

\usepackage[
	a4paper,
	margin=1in
]{geometry}
\usepackage{setspace}
\usepackage{sectsty} 
\usepackage{titlesec}
\usepackage{appendix}
\usepackage{hyphenat}

\usepackage{amsmath, amssymb, amsthm}
\usepackage{mathtools}

\usepackage[
	setpagesize=false,
	colorlinks=true,
	linkcolor=BrickRed,
        citecolor=OliveGreen,
        urlcolor=black,
	pdfencoding=auto,
	psdextra,
]{hyperref}
\usepackage{cleveref}

\usepackage{etoolbox} 
\usepackage[shortlabels]{enumitem}
\usepackage{xparse} 

\usepackage{graphicx}
\usepackage[dvipsnames]{xcolor}
\usepackage{tikz}

\usepackage{todonotes}
\usepackage[
    deletedmarkup=sout,
    commentmarkup=uwave,
    authormarkuptext=name
]{changes}

\usepackage{comment}

\setlist[enumerate,1]{label=(\arabic*), ref=(\arabic*)}
\setlist[enumerate,3]{label=(\roman*), ref=(\roman*)}

\allsectionsfont{\boldmath} 

\titlelabel{\thetitle.\quad}

\usepackage{microtype}
\theoremstyle{plain}
\newtheorem{theorem}{Theorem}[section]
\newtheorem{lemma}[theorem]{Lemma}
\newtheorem{corollary}[theorem]{Corollary}
\newtheorem{proposition}[theorem]{Proposition}

\newtheorem{conjecture}[theorem]{Conjecture}

\newtheorem{claim}[theorem]{Claim}
\newtheorem*{claim*}{Claim}

\makeatletter
\newenvironment{claimproof}[1][Proof]{\par
	\pushQED{\qed}%
	
	\normalfont \topsep6\p@\@plus6\p@\relax
	\trivlist
	\item[\hskip\labelsep
	\textit{#1}\@addpunct{.}~]\ignorespaces
}{%
	\popQED\endtrivlist\@endpefalse
}
\makeatother

\newtheorem{case}{Case}
\makeatletter
\@addtoreset{case}{section}
\makeatother

\theoremstyle{definition}
\newtheorem{definition}[theorem]{Definition}
\newtheorem*{definition*}{Definition}
\newtheorem{remark}[theorem]{Remark}

\newcommand{\calA}{\mathcal{A}}
\newcommand{\calB}{\mathcal{B}}
\newcommand{\calC}{\mathcal{C}}

\newcommand{\calP}{\mathcal{P}}

\newcommand{\calW}{\mathcal{W}}

\newcommand{\ve}{\varepsilon}

\newcommand{\defeq}{\coloneqq}

\title{Tower heights for color-avoiding Ramsey numbers of monotone paths}

\author{Jigang Choi%
        \thanks{Department of Mathematical Sciences, KAIST, South Korea. \emph{E-mails:} \textbf{\{jigang.choi, hyunwoo.lee\}@kaist.ac.kr}. Supported by the National Research Foundation of Korea (NRF) grant funded by the Korea government (MSIT) No. RS-2023-00210430.} \and
        Hyunwoo Lee~\footnotemark[1]~\thanks{Extremal Combinatorics and Probability Group (ECOPRO), Institute for Basic Science (IBS). Supported by  the Institute for Basic Science (IBS-R029-C4).}\and
        Tuan Tran~\thanks{School of Mathematical Sciences, University of Science and Technology of China. Supported by the Excellent Young Talents Program (Overseas) of the National Natural Science Foundation of China under Grant No. GG0010007003.}
}

\usepackage[square,sort,comma,numbers]{natbib}
\begin{document}
\date{\today}
\maketitle

\begin{abstract}

    Ramsey numbers of monotone paths in ordered hypergraphs form a natural higher-uniformity extension of the classical Erd\H{o}s--Szekeres theorems, and their tower height was determined by Moshkovitz and Shapira. A color-avoiding variant, initiated by Loh and further developed by Gowers and Long and by Mulrenin, Pohoata, and Zakharov, asks for monotone paths whose edges use only a bounded number of colors rather than a single color.

    For integers $q>p$, let $A_k(n;q,p)$ be the least integer $N$ such that every $q$-coloring of the ordered complete $k$-uniform hypergraph on $\{1,\ldots,N\}$ contains a monotone path of length $n$ whose edges use at most $p$ colors. We prove that, for every fixed $p$ and all sufficiently large $q$, the exact tower height of $A_k(n;q,p)$ is $\lceil (k-1)/p\rceil$. Thus the number of colors allowed on the path affects the Ramsey number at the level of tower height: allowing $p$ colors lowers the height from $k-1$ in the monochromatic problem to $\lceil (k-1)/p\rceil$. This answers questions of Mulrenin, Pohoata, and Zakharov.

    The upper bound follows from a simple block-compression argument. The main contribution is the matching lower bound, for which we develop a novel variant of the stepping-up method. A surprising feature of the proof is the appearance of the Morse--Hedlund theorem, a foundational result in symbolic dynamics and combinatorics on words. We establish and use a finite version of this theorem, which may be of independent interest.
\end{abstract}


\section{Introduction}\label{sec:intro}

In their classical 1935 paper, Erd\H{o}s and Szekeres~\cite{ErdosSzekeres}
proved two foundational results that have become central in Ramsey theory
and discrete geometry. One of them is the monotone subsequence theorem,
which asserts that every sequence of $n^2+1$ distinct real numbers
contains a monotone subsequence of length $n+1$. The second is the
cup--cap theorem, which states that every set of
$\binom{2n}{n}+1$ points in the plane, in general position and with
distinct $x$-coordinates, contains $n+2$ points
$p_1,\ldots,p_{n+2}$, ordered by increasing $x$-coordinate, such that
the slopes of the consecutive segments
$(p_1,p_2),(p_2,p_3),\ldots,(p_{n+1},p_{n+2})$
are monotone. These two theorems have had numerous applications; see,
for example, Steele~\cite{Steele} for several proofs of the monotone
subsequence theorem and related variations. 

Fox, Pach, Sudakov, and Suk~\cite{Fox-Pach-Sudakov-Suk} observed that
these two classical results fit naturally into a common Ramsey-theoretic
framework involving monotone paths in ordered hypergraphs. Let
$[N]^{(k)}$ denote the complete $k$-uniform hypergraph on the ordered
vertex set $[N]=\{1,2,\ldots,N\}$.
Given vertices $v_1<v_2<\cdots<v_{n+k-1}$, we say that the edges
\[
    \{v_1,\ldots,v_k\}, \{v_2,\ldots,v_{k+1}\}, \ldots, \{v_n,\ldots,v_{n+k-1}\}
\]
form a \emph{monotone path of length $n$}.  We denote this ordered
path by $P_n^{(k)}$.

Let $MS_k(n;q)$ be the least integer $N$ such that every $q$-coloring of
the edges of $[N]^{(k)}$ contains a monochromatic copy of $P_n^{(k)}$.
With this convention, the monotone subsequence theorem and the
cup--cap theorem of Erd\H{o}s and Szekeres give the classical bounds
\[
    MS_2(n;2) \le n^2+1
    \qquad\text{and}\qquad
    MS_3(n;2) \le \binom{2n}{n}+1,
\]
respectively. Motivated in part by geometric generalizations of the
cup--cap theorem, Fox, Pach, Sudakov, and Suk~\cite{Fox-Pach-Sudakov-Suk}
initiated the systematic study of $MS_k(n;q)$ for general $k$ and $q$.  This problem was subsequently resolved
by Moshkovitz and Shapira~\cite{MS}, who established a connection
between these Ramsey numbers and the enumeration of high-dimensional
integer partitions, or equivalently antichains in generalized
hypercubes. To state their estimates, we define the tower function
$T_i(x)$ recursively by 
$T_1(x)=x$ and $T_{i+1}(x)=2^{T_i(x)}$. We say that $i$ is the height of the tower function $T_i(x)$.

\begin{theorem}[Moshkovitz--Shapira~\cite{MS}]\label{thm:M-S}
    There is an absolute constant $n_0$ such that for all $k \geq 3$, $q \geq 2$ and $n\ge n_0$, we have 
    $$
        T_{k-1}\left(\frac{n^{q-1}}{2\sqrt{q}}\right)\leq MS_k(n; q) \leq T_{k-1}\left(2n^{q-1}\right).
    $$
\end{theorem}

We also mention that Pohoata and Zakharov~\cite{Pohoata-Zakharov}
recently showed that, when $q$ is large, the lower bound in
Theorem~\ref{thm:M-S} has the correct order of magnitude in the top
exponent:
\[
    MS_k(n;q)
    =
    T_{k-1}\left(\Theta\left(\frac{n^{q-1}}{\sqrt q}\right)\right).
\]
See also~\cite{Falgas-Ravry-Raty-Tomon} for related developments.

The present paper studies a color-avoiding analogue of these Ramsey
numbers. Instead of looking for a monochromatic path, we ask for a
monotone path whose edges use only a bounded number of colors. This
point of view generalizes the usual Ramsey problem: in a $q$-coloring,
a subgraph using at most one color is monochromatic, while a subgraph
using at most $p$ colors avoids at least $q-p$ colors. Color-avoidance
questions of this type go back to Erd\H{o}s, Hajnal, and Rado~\cite{StepUp}
and were further developed by Erd\H{o}s and Szemer\'edi~\cite{Erdos-Szemeredi}.
For more recent work in this direction, we refer the reader to
\cite{set-ramsey-lower,set-ramsey-code,Loh,Gowers-Long,
Dubroff-Girao-Hurley-Yap,Mulrenin-Pohoata-Zakharov,
Fox-Sudakov-Wigderson}.

Following~\cite{Mulrenin-Pohoata-Zakharov}, let $A_k(n;q,p)$ denote
the least integer $N$ such that every $q$-coloring of $[N]^{(k)}$
contains a monotone path of length $n$ whose edges receive at most
$p$ colors. Thus $A_k(n;q,1)=MS_k(n;q)$. Already the first nontrivial case,
$A_2(n;3,2)$, is quite subtle.

In the graph case, a simple application of the Erd\H{o}s--Szekeres
monotone subsequence theorem gives $A_2(n;3,2) \le n^2+1$. Loh~\cite{Loh} observed that $A_2(n;3,2)\ge n^{3/2}$ and, using
Seidenberg's proof~\cite{Seidenberg} of the Erd\H{o}s--Szekeres theorem
together with the triangle-removal lemma of Ruzsa and
Szemer\'edi~\cite{Ruzsa-Szemeredi}, proved the first nontrivial upper
bound $A_2(n;3,2)=o(n^2)$. Gowers and Long~\cite{Gowers-Long} later improved this by showing that
there is an absolute constant $\varepsilon>0$ such that
$A_2(n;3,2)\le n^{2-\varepsilon}$,
and they conjectured that
$A_2(n;3,2)=\Theta(n^{3/2})$.

For higher uniformities, there is a simple general upper bound obtained
by grouping colors. Indeed, partition the $q$ colors into
$\lceil q/p\rceil$ classes, each containing at most $p$ colors. A
monochromatic monotone path in the induced $r$-coloring then corresponds
to a monotone path using at most $p$ of the original colors. Hence, for
$k\ge 3$,
\[
    A_k(n;q,p)
    \le
    MS_k(n;\lceil q/p\rceil)
    \le
    T_{k-1}\left(2n^{\lceil q/p\rceil}\right).
\]
This bound has the same tower height $k-1$ as the monochromatic problem.
Mulrenin, Pohoata, and Zakharov~\cite{Mulrenin-Pohoata-Zakharov} showed
that, when $p>q/2$, this tower height can be reduced by one.

\begin{theorem}[Mulrenin--Pohoata--Zakharov~\cite{Mulrenin-Pohoata-Zakharov}]
\label{thm:color-upperbound}
For all integers $n\ge k\ge 3$ and $q>p\ge 1$ with $p>q/2$, one has
\[
    A_k(n;q,p)
    \le
    T_{k-2}\left(
        n^{O\left(\binom{q}{p}^2
        \log \binom{q}{p}\right)}
    \right).
\]
\end{theorem}

Mulrenin, Pohoata, and Zakharov~\cite{Mulrenin-Pohoata-Zakharov}
asked for lower bounds for $A_k(n;q,p)$. They also singled out the case
$A_3(n;3,2)$, asking whether the known upper bound could be improved.
At the same time, they pointed out that the classical stepping-up lemma
of Erd\H{o}s and Hajnal~\cite{StepUp}, one of the standard tools for
proving tower-type lower bounds in hypergraph Ramsey theory, is difficult
to apply directly in this color-avoiding setting.

Our main result overcomes this difficulty and determines the correct
tower height of $A_k(n;q,p)$ for every fixed $p$, provided that $q$ is
sufficiently large.

\begin{theorem}\label{thm:main}
The following statements hold.

\begin{itemize}
\item[\rm (i)]
For every $q>p\ge 2$, every $k\ge 2$, and all sufficiently large $n$,
\[
    A_k(n;q,p)
    \le
    T_{\lceil (k-1)/p\rceil}\left(n^{q^p}\right).
\]

\item[\rm (ii)]
For every fixed $p\ge 1$, there exist an integer $q_p$ and a positive
function $\gamma_p(q)$ with $\gamma_p(q)\to\infty$ as $q\to\infty$ such
that the following holds for all $q\ge q_p$. For every $k\ge 2$ and all
sufficiently large $n$,
\[
    A_k(n;q,p)
    \ge
    T_{\lceil (k-1)/p\rceil}\left(n^{\gamma_p(q)}\right).
\]
\end{itemize}
\end{theorem}

Theorem~\ref{thm:main} determines the
tower height of $A_k(n;q,p)$ for every fixed $p$ and all sufficiently
large $q$: it is exactly $\lceil (k-1)/p\rceil$.
Thus allowing a monotone path to use $p$ colors lowers the tower height
from $k-1$, the height in the monochromatic problem, to
$\lceil (k-1)/p\rceil$. The upper bound in
Theorem~\ref{thm:main}(i) follows from a short block-compression
argument. The main contribution of the paper is the matching lower bound
in Theorem~\ref{thm:main}(ii).

To prove this lower bound, we develop a new stepping-up method for
monotone paths of bounded color complexity. In the course of the proof,
we establish a finite analogue of the Morse--Hedlund
theorem~\cite{MorseHedlund}, a foundational result in symbolic dynamics
and combinatorics on words. The classical Morse--Hedlund theorem says
that if a one-sided infinite word has at most $m$ distinct contiguous
subwords of length $m$, for some $m$, then the word is eventually
periodic; see, for example, \cite[Theorem~4.3.1]{LindMarcus}. Our finite
version gives a robust periodicity principle for finite words, and this
principle is what makes the stepping-up construction possible. We expect
this finite Morse--Hedlund theorem to be useful in other problems
involving words and ordered combinatorial structures.


\subsection{High-level overview}\label{subsec:overview}

We now give an overview of the main ideas of the proof. The upper bound
is a short \emph{block-compression} argument. We partition the ordered
vertex set into blocks of size $p$ and define an auxiliary coloring on a
lower-uniformity ordered hypergraph. An edge of the auxiliary hypergraph
records the colors of $p$ consecutive $k$-edges in the original
hypergraph. A monochromatic monotone path in the auxiliary coloring then
expands to a monotone path in the original coloring whose edges use at
most $p$ colors. This reduces the relevant uniformity from $k$ to
$\lceil (k-1)/p\rceil+1$, giving the upper bound in
Theorem~\ref{thm:main}(i).

The main contribution of the paper is the matching lower bound in
Theorem~\ref{thm:main}(ii). Its proof is based on a color-avoiding
variant of the stepping-up lemma. As in the classical stepping-up
construction, an ordered edge in the binary cube is associated with a
\emph{$\delta$-sequence}, defined in Section~\ref{sec:Prelim}. When this
$\delta$-sequence is monotone, the coloring behaves essentially as in
the usual stepping-up lemma. Thus a long monotone path using few colors
would give a forbidden lower-uniformity monotone path. The main new
difficulty is the non-monotone case.

In the non-monotone case, our coloring records more local information
than in the classical construction. In particular, it records the local
pattern of the $\delta$-sequence, as well as auxiliary colors coming from
several lower-uniformity colorings. This information is arranged so that
if a long monotone path uses at most $p$ colors, then the associated
sequence of local patterns has small complexity. At this point we use a
finite analogue of the Morse--Hedlund theorem, which we prove in
Section~\ref{sec:finite-Morse-Hedlund}. Roughly speaking, this finite
Morse--Hedlund theorem says that a sufficiently long finite sequence with
few distinct contiguous blocks must become periodic after deleting only
a bounded number of terms from its two ends.

Applied to the sequence of local patterns arising from a hypothetical
color-avoiding monotone path, the finite Morse--Hedlund theorem forces a
periodic structure with some least period $d\le p$. This periodicity has
two crucial consequences. First, different residue classes modulo $d$
give genuinely different local patterns. Second, along at least one
residue class, the associated $\delta$-values form a strictly monotone
subsequence. The coloring is designed to record all possible phases of
this periodic structure, so that whichever period and phase arise, the
strictly monotone subsequence can be converted into a forbidden
lower-uniformity monotone path using strictly fewer colors.

Combining the monotone and non-monotone cases yields the stepping-up
lemma stated as Theorem~\ref{thm:main-stepup}. In the monotone case, the
uniformity decreases by one, as in the classical stepping-up lemma. In
the non-monotone case, the uniformity decreases by a factor of about
$p$, while the allowed color complexity also decreases. Under iteration, one therefore either applies the stepping-up lemma
repeatedly, raising the tower height, or eventually reaches color
complexity $1$, where Remark~\ref{rem:p=1-main-stepup} gives the usual
monochromatic stepping-up argument.

There is one additional endpoint needed to start the iteration. Theorem
\ref{thm:stepup-low-uniformity} steps up from a graph construction to
uniformity $p+2$ and gives the first exponential lower bound. The
all-phases stepping-up lemma, Theorem~\ref{thm:main-stepup}, then
applies from uniformity $2p+2$ onward and raises the tower height
thereafter. The intermediate uniformities are handled by monotonicity in
the uniformity. Together, these ingredients give the optimal tower-type
lower bounds in Theorem~\ref{thm:main}(ii).

\medskip
\noindent\textbf{Organization of the paper.}
The rest of the paper is organized as follows. In
Section~\ref{sec:Prelim}, we collect notation and recall the basic properties of the classical stepping-up construction that will be used
throughout the paper. In Section~\ref{sec:upper-proof}, we prove the upper bound in Theorem~\ref{thm:main}(i). In Section~\ref{sec:finite-Morse-Hedlund}, we prove the finite Morse--Hedlund theorem and derive the structural consequence for $\delta$-sequences needed later. Section~\ref{sec:step-up} contains the
two stepping-up lemmas that form the core of the lower-bound argument. In Section~\ref{sec:lower-bound}, we combine these stepping-up lemmas
with the graph lower bound and monotonicity in the uniformity to prove Theorem~\ref{thm:main}(ii). We finish with concluding remarks and open
problems.

\section{Preliminaries}\label{sec:Prelim}

\subsection{Notation}

Throughout the remaining paper, we use standard notation in graph theory. For the purpose of finding a monotone path, which is a certain ordered hypergraph on distinct integers, we identify hyperedges as sequences of integers if necessary. Then we distinguish hyperedges using the information in the corresponding sequences. To this end, we define several notations for sequences.  

Let $S$ be a sequence over an alphabet $\calA$, not necessarily finite. For a positive integer $p$, we say a contiguous subsequence of length $p$ in $S$ is a \emph{$p$-block} of $S$ and write $\calB(S; p)$ for the set of $p$-blocks in $S$.
We denote by $\calC(S; p)$ the number of distinct $p$-blocks in $S$, equivalently the cardinality of $\calB(S; p)$, and call it \emph{$p$-complexity} of $S$. For instance, if $S = (a, b, 1, c, b, 1)$ then $\calB(S; 2) = \{(a, b), (b, 1), (1, c), (c, b)\}$ and $\calC(S; 2) = 4$. 

For a sequence $X = (x_1,\dots,x_p)$ over a totally ordered alphabet (e.g., the integers), we define the \emph{pattern} of $X$ to be the tuple $(\ve_{ij})_{1\leq i<j\leq p}$, where $\ve_{ij}$ is the unique symbol in $\{<,=,>\}$ satisfying $x_i\,\ve_{ij}\,x_j$. For example, $(4,7,7,2)$ and $(3,5,5,1)$ have the same pattern. 
Let $S = (s_1, s_2, \dots, )$ be a sequence over a totally ordered alphabet and let $p$ be a positive integer. We say a sequence $ (P_1, P_2, \dots, )$ is a \emph{$p$-pattern sequence} of $S$ if $P_i$ is a pattern of $(s_i, s_{i+1}, \dots, s_{i+p-1})$ for each $i$. Also, we denote by $\calP(S;p)$ the number of distinct patterns in the $p$-pattern sequence of $S$.
For example, the $2$-pattern sequence of $S=(1, 2, 2, 4, 3)$ is $( (<), (=), (<), (>) )$ and $\calP(S;2)=3$. 

Lastly, for a hypergraph $H$ and a set $M$ of size $q$, we say $\chi: E(H)\to M$ is a \emph{$q$-coloring} of $H$ and $M$ is the \emph{palette} of $\chi$.

\subsection{The classical stepping-up lemma} 

Since the lower-bound argument in this paper is a color-avoiding variant of the classical stepping-up method of Erd\H{o}s and Hajnal~\cite{StepUp}, we begin by recalling the basic setup of the stepping-up construction and
the properties of the associated $\delta$-sequences that will be used later.

The stepping-up lemma is a fundamental tool for deriving lower bounds on hypergraph Ramsey numbers from lower-uniformity constructions. Let
$R_k(t;q)$ denote the $q$-color Ramsey number of the complete $k$-uniform hypergraph on $t$ vertices; equivalently, $R_k(t;q)$ is the
least integer $m$ such that every $q$-coloring of $[m]^{(k)}$ contains a monochromatic copy of $[t]^{(k)}$. The Erd\H{o}s--Hajnal stepping-up
lemma~\cite{StepUp} states that, for $t\ge k\ge3$ and $q\ge2$,
\begin{equation}\label{eq:classic-stepup}
    R_{k+1}(2t+k-4;q) \ge 2^{R_k(t;q)-1}.
\end{equation}

We next recall the notation underlying the proof of \eqref{eq:classic-stepup}. This notation, based on $\delta$-sequences, is central to the stepping-up construction and to many of its variants.

\begin{definition}
    For an integer $m \geq 1$ and two distinct vectors $u, v \in \{0, 1\}^m$, define
    $$
        \delta(u, v) \defeq \max\{i\in [m]: (u)_i \neq (v)_i\}.
    $$
\end{definition}

Here and throughout the rest of the paper, we equip $\{0,1\}^m$ with the reverse lexicographic order, denoted by $<$. Thus, for two distinct vectors $u,v\in \{0,1\}^m$, we write $u<v$ if and only if $(u)_{\delta} < (v)_{\delta}$, where $\delta = \delta(u, v)$. With respect to this order, a sequence of binary vectors $v_1,\ldots,v_{\ell}$ is called \emph{monotone} if
$v_1 < \cdots < v_{\ell}$ or $v_1 > \cdots > v_{\ell}$.

For a sequence of vectors $S =v_1,\ldots, v_{\ell}$ in $\{0, 1\}^m$, with no two consecutive vectors equal, we define $\Delta(S) = \{\delta_i\}_{i=1}^{\ell - 1}$, where $\delta_i = \delta(v_i, v_{i+1})$ for each $i \in [\ell - 1]$. We call $\Delta(S)$ the \emph{$\delta$-sequence} of $S$. For $V=\{v_1,\dots,v_k\}\subseteq\{0,1\}^m$, let $\overrightarrow{V}=(v_{i_1},\dots,v_{i_{k}})$ be the sequence obtained by arranging the elements of $V$ in increasing order, that is, $v_{i_1}<v_{i_2}<\dots<v_{i_k}$. We also write $\Delta(V)$ for $\Delta(\overrightarrow{V})$.

The basic idea behind the proof of \eqref{eq:classic-stepup} is the
following. Set $m=R_k(t;q)-1$, and fix a $q$-coloring $\chi$ of $[m]^{(k)}$ with no monochromatic copy of $[t]^{(k)}$. Given a
$(k+1)$-edge in $\{0,1\}^m$, order its vertices increasingly and consider the associated $\delta$-sequence. Since the entries of this
sequence lie in $[m]$, one can color the $(k+1)$-edge using the $\delta$-sequence together with the coloring $\chi$. With an appropriate
choice of the coloring rule, the resulting coloring of $(\{0,1\}^m)^{(k+1)}$ contains no monochromatic copy of $[2t+k-4]^{(k+1)}$.

Many hypergraph Ramsey lower bounds are obtained by variants of this construction. A key
property used in such arguments is that $\delta$-sequences arising from monotone sequences of binary vectors have the \emph{unique maximum
property}: every contiguous subsequence has a unique maximum.

\begin{proposition}\label{prop:UMP}
    Let $S = \{v_i\}_{i=1}^{\ell}$ be a monotone sequence of vectors in $\{0, 1\}^m$. Then the $\delta$-sequence of $S$ has the unique maximum property.
\end{proposition}

Proposition~\ref{prop:UMP} is standard in hypergraph Ramsey theory and
will play a crucial role in the proofs of our stepping-up lemmas,
Theorems~\ref{thm:stepup-low-uniformity} and~\ref{thm:main-stepup}.
For a proof, we refer the reader to~\cite{UMP}.


\section{Proof of the upper bound}\label{sec:upper-proof}

In this section, we prove Theorem~\ref{thm:main}(i). The argument is a simple block-compression: we group consecutive vertices into blocks of
size $p$, and let each edge of a lower-uniformity auxiliary hypergraph
record the colors of $p$ consecutive $k$-edges in the original hypergraph. The following lemma gives the precise reduction.

\begin{lemma}\label{lem:block-compression}
Let $k,q,p,n$ be positive integers with $k\geq 2$. Then 
$$
    A_k(n;q,p)\leq p\, MS_{\lceil (k-1)/p\rceil+1}(\lceil n/p\rceil; q^p).
$$
\end{lemma}

\begin{proof}[Proof of Lemma~\ref{lem:block-compression}]
    Let $r = \lceil (k-1)/p\rceil+1$, $m = \lceil n/p\rceil$, and $M = MS_{r}(m;q^p)$. Suppose we are given an arbitrary $q$-coloring $\chi$ of $[pM]^{(k)}$. We will find a monotone path of length $n$ using at most $p$ colors.

    Split $[pM]$ into $M$  blocks $B_1,\ldots,B_M$, each of size $p$, where $B_i=\{(i-1)p+1,\ldots, ip\}$. We define a new coloring $\psi$ on $[M]^{(r)}$. Take an $r$-edge $\{i_1<\cdots<i_r\}$, and list the vertices in $B_{i_1}\cup\cdots\cup B_{i_r}$ as $x_{1}<\cdots<x_{pr}$. The choice of $r$ gives $pr\ge p+k-1$.  Therefore the $p$ consecutive $k$-sets $\{x_{1},\ldots,x_{k}\}, \ldots , \{x_{p},\ldots,x_{p+k-1}\}$ all fit inside this union of $r$ blocks.  Color the $r$-edge $\{i_1<\cdots<i_r\}$ by the ordered $p$-tuple $(\chi(x_{1}, \ldots,x_{k}),\ldots, \chi(x_{p}, \ldots,x_{p+k-1}))$.  This is a coloring with at most $q^p$ colors.

    By the definition of $M$, the coloring $\psi$ contains a monochromatic monotone path of length $m$.  Thus there are block indices $i_1<\cdots<i_{m+r-1}$ such that every consecutive $r$-tuple $i_j,\ldots,i_{j+r-1}$ has the same $\psi$-color. Call this common color $(c_1,\ldots,c_p)$.

    Now expand the selected blocks back into the vertices of the original hypergraph.  List all vertices in $B_{i_1}\cup\cdots\cup B_{i_{m+r-1}}$ as $y_1<\cdots<y_{p(m+r-1)}$.  We claim that the first $pm$ consecutive $k$-edges in this $y$-sequence use only the colors $c_1,\ldots,c_p$. Indeed, an edge starting at position $p(j-1)+t$, with $1\le j\le m$ and $1\le t\le p$, lies inside the union of the $r$ blocks $B_{i_j},\ldots,B_{i_{j+r-1}}$, because $t+k-1\le p+k-1\le pr$.  By the definition of $\psi$, its color is $c_t$.
    Therefore, these vertices form a monotone path of length $pm$ using only the colors $c_1,\ldots,c_p$. Since $pm\ge n$, the first $n$ edges give the desired path. Hence $A_k(n;q,p)\le pM$. This completes the proof.
\end{proof}

\begin{proof}[Proof of Theorem~\ref{thm:main}(i)]
    Put $r=\lceil(k-1)/p\rceil+1$ and $m=\lceil n/p\rceil$.  By Lemma~\ref{lem:block-compression}, we have $A_k(n;q,p)\leq p\,MS_r(m;q^p)$. If $r=2$, the graph Erd\H{o}s--Szekeres bound gives $MS_2(m;q^p)\leq m^{q^p}+1$, and hence $A_k(n;q,p)\leq n^{q^p}$ for all sufficiently large $n$. If $r \geq 3$, Theorem~\ref{thm:M-S} gives $MS_r(m;q^p)\leq T_{r-1}(2m^{q^p})$ for all sufficiently large $m$.
    Since $r-1=\lceil(k-1)/p\rceil$, the factor $p$ and the constants inside the bottom of the tower are harmless for sufficiently large $n$. Therefore $A_k(n;q,p)\le T_{\lceil(k-1)/p\rceil}(n^{q^p})$. This completes the proof.
\end{proof}

\section{Finite Morse--Hedlund theorem}\label{sec:finite-Morse-Hedlund}

In this section, we introduce and prove a finite analogue of the
celebrated Morse--Hedlund theorem. For a positive integer $r$, an
infinite sequence $S=\{s_i\}_{i\in\mathbb N}$ is \emph{periodic with
period $r$} if $s_i=s_{i+r}$ for all $i$. Similarly, a finite sequence
$S=\{s_i\}_{i=1}^n$ is \emph{periodic with period $r$} if
$s_i=s_{i+r}$ for all $i\in[n-r]$. We say that an infinite sequence is
\emph{eventually periodic} if deleting finitely many initial terms leaves
a periodic sequence.

\begin{theorem}[Morse--Hedlund theorem~\cite{MorseHedlund}]\label{thm:M-H}
    Let $S$ be an infinite sequence over a finite alphabet. Then $S$ is eventually periodic if and only if there exists a positive integer $p$ such that $\calC(S;p)\leq p$.
\end{theorem}

We prove the following finite counterpart of Theorem~\ref{thm:M-H}. In
the infinite setting, small block complexity forces eventual
periodicity. For finite sequences, endpoint effects cannot be avoided:
even if the number of distinct blocks is small, one may have to delete
terms from both ends before periodicity appears. Theorem~\ref{thm:finite-M-H}
shows that this is the only obstruction, with a bound depending only on
the complexity parameter.


Let $a,b$ be non-negative integers and let $S$ be a finite sequence over
an alphabet. We define an \emph{$(a,b)$-reduction} of $S$ to be the
sequence obtained from $S$ by deleting the first $a$ elements and the
last $b$ elements of $S$.

\begin{theorem}\label{thm:finite-M-H}
Let $n, m, p$ be positive integers such that $n \geq p + m$ and $m \geq p$. Let $S=\{s_i\}_{i=1}^n$ be a finite sequence over an alphabet. If $\calC(S;m)\le p$, then there exist non-negative integers $a, b \leq 2p$ such that the $(a, b)$-reduction of $S$ is periodic with period at most $p$. 
\end{theorem}

The need to delete terms from both ends is a genuine finite phenomenon,
rather than an artifact of the proof. For instance, let
$S=(1,2,2,\ldots,2,1)$, with sufficiently many copies of $2$ in the
middle. Then $\calC(S;3)=3$, but the terminal $1$ prevents any long
suffix obtained by deleting only initial terms from being periodic with
period at most $3$. Thus, in a finite version of Morse--Hedlund, one
must allow endpoint effects at both ends of the sequence.


We next indicate the main difficulty in adapting the usual proof of the
Morse--Hedlund theorem. For an infinite sequence $S$, one always has
$\calC(S;m)\le \calC(S;m+1)$
for every $m\ge1$. Hence, if $\calC(S;p)\le p$ and
$\calC(S;1)\ge2$, then
$2\le \calC(S;1)\le \calC(S;2)\le \cdots \le \calC(S;p)\le p$.
By the pigeonhole principle, there is some $\ell<p$ such that
$\calC(S;\ell)=\calC(S;\ell+1)$. This equality implies that every
$\ell$-block has a unique right extension, and this unique-extension
property is the mechanism that forces eventual periodicity.

For finite sequences, both parts of this argument can fail. First,
$\calC(S;\ell)$ need not be nondecreasing in $\ell$: for example, if
$S=(1,2,\ldots,n)$, then $\calC(S;1)=n$ whereas
$\calC(S;2)=n-1$. Second, even equality
$\calC(S;\ell)=\calC(S;\ell+1)$ does not by itself imply unique
extension, because of endpoint effects. For instance, for
$S=(1,1,1,2)$ we have $\calC(S;2)=\calC(S;3)=2$, but the block
$(1,1)$ has two possible right extensions, namely $(1,1,1)$ and
$(1,1,2)$.

The proof below overcomes these finite obstructions by encoding the
$\ell$-blocks of $S$ as vertices of a directed graph and the
$(\ell+1)$-blocks as directed edges. The sequence $S$ then determines a
walk in this graph. The assumption $\calC(S;m)\le p$ forces, for some
$\ell<m$, a situation in which this walk is essentially cyclic, except
possibly near its two endpoints. Removing a bounded number of terms from
the two ends eliminates these endpoint effects and leaves a periodic
sequence of period at most $p$.

\begin{proof}[Proof of Theorem~\ref{thm:finite-M-H}]

If $m = 1$, the sequence $S$ consists of a single alphabet and its $(0, 0)$-reduction is periodic. Thus, we may assume that $m > 1$.

Despite the inequality $\calC(X; \ell) \leq \calC(X; \ell+1)$ not holding for general finite sequences $X$, we claim that under the assumption $\calC(S; m) \leq p$, the inequality holds for all $\ell \in [m]$.
For each $\ell \in [m]$, we define a digraph $D_{\ell}$, which allows loops, on $\calB(S; \ell)$ in which for $x,y\in \calB(S; \ell)$ there is an arc from $x$ to $y$ if and only if there is an $(\ell + 1)$-block $z\in \calB(S; \ell+1)$ whose prefix of length $\ell$ and suffix of length $\ell$ form $x$ and $y$, respectively. Also, we define a directed walk $\calW_{\ell} \defeq (w^{(\ell)}_1, w^{(\ell)}_2, \dots, w^{(\ell)}_{n-\ell + 1})$ on $D_{\ell}$ such that $w^{(\ell)}_i$ is the $\ell$-block $(s_{i}, \dots, s_{i + \ell - 1})$ for each $i\in [n - \ell + 1]$. We remark that $\calW_{\ell}$ is indeed a directed walk, that is, $w^{(\ell)}_i w^{(\ell)}_{i+1} \in E(D_{\ell})$ for each $i\in [n-\ell]$ by the definition of $D_{\ell}$.  

\begin{claim}\label{clm:monotonic}
    For any $\ell \in [m-1]$, we have $\calC(S; \ell) \leq \calC(S; \ell+1)$.   
\end{claim}

We note that from its definition, it holds that $|V(D_{\ell})| = \calC(S; \ell)$ and $|E(D_{\ell})| = \calC(S; \ell+1)$. Hence, it suffices to show that $|V(D_{\ell})| \leq |E(D_{\ell})|$.

\begin{claimproof}[Proof of Claim~\ref{clm:monotonic}]
    Observe that the sequence $S$ gives a walk on $D_{\ell}$. For example, if $S=(1,2,3,4)$ and $\ell = 2$, then $S$ gives a directed walk $(1, 2) \rightarrow (2, 3) \rightarrow (3, 4)$ on $D_2$. Furthermore, this walk traverses every vertex and edge of $D_{\ell}$. Consequently, the underlying undirected graph of $D_{\ell}$ is connected. Therefore, to show $|E(D_\ell)| \geq |V(D_\ell)|$, it suffices to prove that $D_\ell$ contains at least one directed cycle.
    
    Observe that the trivial inequality $\calC(S; t)\leq \calC(S; t+1)+1$ holds for any integer $t$, since every $t$-block in $S$, except possibly the very last one, extends to at least one $(t+1)$-block. Applying this inequality repeatedly from $k= \ell$ to $m-1$, we obtain 
    $$\calC(S; \ell) \leq \calC(S;m) + (m-\ell).$$
    Together with the assumption $\calC(S; m) \leq p$, it holds that 
    \begin{equation}\label{eq:bound-VDell}
        |V(D_{\ell})| = \calC(S; \ell) \leq p + m - \ell  < n - \ell + 1.
    \end{equation}

    Recall that there is a directed walk $\calW_{\ell} = (w^{(\ell)}_1, \dots, w^{(\ell)}_{n-\ell + 1})$ on $D_{\ell}$. By \eqref{eq:bound-VDell} and the pigeonhole principle, there exist distinct indices $i, j$ such that $w^{(\ell)}_i = w^{(\ell)}_j$. This implies that $\calW_{\ell}$ induces at least one directed cycle between $w_i$ and $w_j$. This completes the proof.
\end{claimproof}

If $\calC(S; 1) = 1$, then $S$ consists of a single alphabet, and hence the conclusion follows immediately. Thus, we may assume that $2 \leq \calC(S; 1)$. Together with Claim~\ref{clm:monotonic}, we deduce 
$$
    2 \leq \calC(S; 1) \leq \calC(S; 2) \leq \cdots \leq \calC(S; m) \leq p.
$$
By the pigeonhole principle, there exists $\ell < m$ such that $\calC(S; \ell) = \calC(S; \ell + 1)$. Then we have 

\begin{equation}\label{eq:double-counting}
    |V(D_\ell)|=\calC(S;\ell)=\calC(S;\ell+1)=|E(D_l)| = \sum_{v\in V(D_{\ell})}d_{D_{\ell}}^+(v),
\end{equation}
where $d_{D_{\ell}}^+(v)$ denotes the out-degree of $v$ in $D_{\ell}$.
We note that every vertex of $D_{\ell}$, except possibly the last vertex of $\calW_{\ell}$, has out-degree of at least $1$. This observation and \eqref{eq:double-counting} imply that there are only two possible cases: either every vertex of $D_{\ell}$ has an out-degree of exactly $1$, or the last vertex of $\calW_{\ell}$ has out-degree $0$, and all the other vertices except one have out-degree $1$.
We divide the proof into these two cases and handle each separately.

\begin{case}\label{case:1}
    Every vertex of $D_{\ell}$ has out-degree exactly one.
\end{case}

By the pigeonhole principle and the fact that $|V(D_{\ell})| \leq p$, there exist indices $i$ and $j$ such that $1\leq i < j \leq p+1$ and $w^{(\ell)}_i = w^{(\ell)}_j$. This implies that $\calW_{\ell}$ induces a directed cycle $C$ between $w^{(\ell)}_i$ and $w^{(\ell)}_j$. Note that the length of the cycle $C$ is at most $p$ as $j - i \leq p$. Fix such $i$ and then choose $j$ to be the smallest among the indices $i < j \leq p+1$ such that $w^{(\ell)}_i = w^{(\ell)}_j$. Since the out-degree of all vertices in $D_{\ell}$ is exactly $1$, the directed cycle $C$ is the unique cycle in $D_{\ell}$. Thus, the walk $\calW_{\ell}$ must repeatedly proceed along $C$ until the sequence terminates once after it meets the vertex $w^{(\ell)}_i$. Hence, by removing the first $i-1$ alphabets of $S$, the remaining subsequence is periodic with period $j-i$, which is at most $p$. In other words, the $(i-1, 0)$-reduction of $S$ is periodic with period at most $p$ and also $i-1 \leq p$. This completes the proof in this case.

\begin{case}
    The last vertex of $\calW_{\ell}$ has out-degree $0$, and all the other vertices except one have out-degree $1$ in $D_{\ell}$.
\end{case}

Let $v$ be the last vertex of $\calW_{\ell}$ and $w$ be the unique vertex of out-degree larger than $1$. 
Let $I \defeq \{i : w^{(\ell)}_i = w\}$ and let $t$ be the maximum element of $I$. Observe that $d^+_{D_{\ell}}(v) = 0$ implies $v$ appears exactly once in $\calW_{\ell}$ as the last vertex of the walk.

Suppose there exist two indices $x < y$ such that $y > t$ and $w^{(\ell)}_{x} = w^{(\ell)}_{y}$. 
Then a sub-walk $\calW_{\ell}$ induces a directed cycle $C$ between $w^{(\ell)}_{x}$ and $w^{(\ell)}_{y}$.
Also, by our choice of $v$ and $t$, we have
\begin{equation}\label{eq:neq}
    v \neq w^{(\ell)}_{x} = w^{(\ell)}_{y} \neq w.
\end{equation}

Note that since $v$ has out-degree zero, the directed cycle $C$ does not contain $v$. If $x > t$, then every vertex in $\calW_{\ell}$ appearing after $w^{(\ell)}_{t}$ has out-degree $1$, except for $v$. Hence, once a directed walk reaches a vertex of $C$, it must continue along $C$ indefinitely. This contradicts the fact that $\calW_{\ell}$ eventually terminates at the vertex $v$. Thus, we may assume that $x < t$. Then the cycle $C$ contains a vertex $w$ and all the other vertices of $C$ have out-degree exactly $1$. 
Thus, for the walk $\calW_{\ell}$ to eventually terminated at $v$, it must exists an index $z$ such that $z > y > t$ and $w^{(\ell)}_z = w$ by \eqref{eq:neq}. This contradicts the fact that $t$ is the maximum element of $I$. Thus, there are no such indices $x$ and $y$.

Therefore, each corresponding vertex of $w^{(\ell)}_{t+1},\dots, w^{(\ell)}_{n - \ell +1}$ appears exactly once in $\calW_{\ell}$ and all of them are distinct. Since $|V(D_{\ell})|\leq p$, we have 
\begin{equation}\label{eq:bound-t}
    (n - \ell + 1) - (t - 1) \leq p.    
\end{equation}

We now consider the subwalk $\calW' = (w^{(\ell)}_1, w^{(\ell)}_2, \dots, w^{(\ell)}_{t})$ induced by the truncated sequence $S' = \{s_i\}_{i=1}^{t+\ell-1}$. Let $D'_{\ell}$ be the sub-digraph of $D_{\ell}$ induced by $\calW'$. Since the path from $w^{(\ell)}_{t}$ to $v$ has been removed, the second outgoing edge from $w$ is no longer used in $\calW'$. Consequently, every vertex in $D'_{\ell}$ has an out-degree of exactly $1$. Applying the argument from Case~\ref{case:1} to $S'$, there exist two non-negative integers $a', b' \leq p$ such that $(a', b')$-reduction of $S'$ is periodic with period at most $p$. As the sequence $S'$ is obtained from removing the last $n - t - \ell + 1$ alphabets, the $(a', b' + n - t - \ell + 1)$-reduction of $S$ is periodic with period at most $p$, and also $b' + n - t - \ell + 1 \leq 2p$ as provided by \eqref{eq:bound-t}. This completes the proof.
\end{proof}

The next lemma packages Theorem~\ref{thm:finite-M-H} in the form needed
for the stepping-up argument: small pattern complexity forces a periodic
structure, and hence a long monotone subsequence along one residue class.


\begin{corollary}\label{cor:monotone-subsequence}
    Let $p\ge1$, $r\ge 2p+1$, and $n\ge 2r+5p$. Let $\Delta=\{\delta_i\}_{i=1}^{n}$ be a sequence of nonnegative integers of length $n$ with the unique maximum property and $\Delta_{2p}=\{\eta_i\}_{i=1}^{n-4p}$ be the $(2p,2p)$-reduction of $\Delta$. If $\calP(\Delta;r)\le p$, there exists $d\in [p]$ such that the following hold.

    \begin{enumerate}[label=(P\arabic*), ref=(P\arabic*)]

        \item\label{it:P1} The $(r-p+1)$-pattern sequence $X=\{x_i\}_{i=1}^{n-3p-r}$ of $\Delta_{2p}$ has least period $d$,

        \item\label{it:P2} the $r$-pattern sequence $Y=\{y_i\}_{i=1}^{n-4p-r+1}$ of $\Delta_{2p}$ satisfies that for $i,j\in [n-4p-(r-1)]$, $y_i\ne y_j$ whenever $i\not\equiv j\pmod d$, and

        \item\label{it:P3} for each $a\in [d]$, the subsequence of $\Delta_{2p}$ given by
        $$
        (\eta_i:1\le i\le n-4p,i \equiv a \pmod d)
        $$
        is either strictly increasing, strictly decreasing, or constant. Furthermore, at least one such subsequence is not constant.
    \end{enumerate}
\end{corollary}

\begin{proof}[Proof of Corollary~\ref{cor:monotone-subsequence}]

Let $X'$ be the $(r-p+1)$-pattern sequence of $\Delta$.
    Note that each $p$-block $(x_i,x_{i+1},\dots,x_{i+p-1})$ of $X'$ is determined by the $r$-pattern of $(\delta_i,\delta_{i+1},\dots,\delta_{i+r-1})$. This implies that if $\calC(X';p)>p$, then $\calP(\Delta;r)>p$, a contradiction. Thus, we have $\calC(X';p) \leq p$. By Theorem~\ref{thm:finite-M-H}, the $(2p,2p)$-reduction $X''$ of $X'$ is periodic with least period $d \leq p$. Observe that $X''$ is the $(r-p+1)$-pattern sequence of $\Delta_{2p}$, i.e., $X''=X$. Therefore, \ref{it:P1} holds. 
    
    We now prove \ref{it:P2}. If $d=1$, then there is nothing to prove, so assume $d\ge2$. Suppose, for contradiction, that
    $i<j\in [n-4p-(r-1)]$, $i\not\equiv j\pmod d$, and $y_i=y_j$. Since $X$ is periodic with least period $d$, there are $(r-p+1)$-blocks $B_0,\ldots,B_{d-1}$ of $\Delta_{2p}$ such that $x_u=B_{u\bmod d}$ for all $u\in [n-3p-r]$. 
    
    We now note that $y_i$ is the $r$-pattern starting at $i$, it determines all the shorter patterns
    $x_i,x_{i+1},\ldots,x_{i+p-1}$. Similarly, $y_j$ determines
    $x_j,x_{j+1},\ldots,x_{j+p-1}$. Hence $y_i=y_j$ implies that
    \begin{equation}\label{eq:B_modd}
    B_{(i+t)\bmod d}=B_{(j+t)\bmod d}    
    \end{equation}
    for every $0\le t\le p-1$. Let $h=(j-i)\bmod d$. Then the equation $(\ref{eq:B_modd})$ is rewritten as
    $$
    B_{(i+t)\bmod d}=B_{(i+t+h)\bmod d}
    $$
    for every $0\le t\le p-1$. Since $d\le p$, the set $\{0,1,\dots,p-1\}$ contains representatives of all residue classes modulo $d$. Therefore, we have
    $$
    B_{u\bmod d}=B_{(u+h)\bmod d}
    $$
    for every $0\le u\le p-1$.
    
    It follows that $X$ is periodic with period $h$. Since $h$ is strictly smaller than $d$, this is a contradiction by the minimality of $d$. Hence \ref{it:P2} holds.

    To prove \ref{it:P3}, let $\Delta_{2p}^a$ denote the subsequence of $\Delta_{2p}$ consisting of all terms $\eta_i$ such that $1\le i\le n - 4p$ and $i \equiv a \pmod d$. It suffices to show that for any $a\in [d]$, $\Delta_{2p}^a$ is either constant or strictly monotone.
    
    Without loss of generality, assume that $\eta_a<\eta_{a+d}$. Then, in $x_a$, the first term is strictly smaller than the $(d+1)$-th term. Since $X$ is periodic with period $d$, in $x_{a+d}$, the first term is strictly smaller than the $(d+1)$-th term. This implies that $\eta_{a+d}<\eta_{a+2d}$. By repeating this process, we obtain the inequality
    $$
    \eta_{a}<\eta_{a+d}<\dots<\eta_{a+td}
    $$ for all $t$ such that $a+td+(r-p)\le n-4p$.
    
    Let $t_0$ be the maximum integer such that $a+t_0d+(r-p)\le n-4p$. By the maximality of $t_0$, we have $a+t_0d+(r-p)> n-4p-d$ which guarantees that
    $$
    a+t_0d> n-4p-d-(r-p)\ge (2r+5p)-r-5p\ge r.
    $$
    Therefore, there is a $j\in [n-4p]$ such that $j+(r-p)=a+t_0d$. Since $\eta_{a+(t_0-1)d}<\eta_{a+t_0d}$, in $x_j$, we know that the last term is strictly larger than the $(r-p+1-d)$-th term. Since $X$ is periodic with period $d$, for any $t>t_0$ such that $a+td\le n-4p$, in $x_{j+(t-t_0)d}$, the last term is strictly larger than the $(r-p+1-d)$-th term. This implies that $\eta_{a+(t-1)d}<\eta_{a+td}$. Hence, we have
    $$
        \eta_{a}<\eta_{a+d}<\dots<\eta_{a+td}
    $$
    for any $t$ such that $a+td\le n-4p$, i.e., $\Delta_{2p}^a$ is strictly monotone.
    
    Observe that if $\eta_a=\eta_{a+d}$, the subsequence $\Delta_{2p}^a$ becomes constant. Finally, if $\Delta_{2p}^a$ is constant for every $a\le d$, it is a contradiction because $\Delta_{2p}$ satisfies the unique maximum property. Therefore, there is at least one integer $a\le d$ such that $\Delta_{2p}^a$ is strictly monotone. 
\end{proof}

\section{Stepping-up lemmas}\label{sec:step-up}

This section contains the two stepping-up ingredients used in the proof
of the lower bound. The first result, Theorem~\ref{thm:stepup-low-uniformity},
is an endpoint stepping-up lemma: it starts from a graph construction and
produces the first exponential lower bound at uniformity $p+2$. The
second result, Theorem~\ref{thm:main-stepup}, is the main stepping-up
lemma. It applies from uniformity $2p+2$ onward and is the place where
the all-phases argument and the finite Morse--Hedlund theorem enter.

We prove Theorem~\ref{thm:stepup-low-uniformity} in
Subsection~\ref{subsec:first-endpoint} and
Theorem~\ref{thm:main-stepup} in Subsection~\ref{subsec:large-range}.
In Section~\ref{sec:lower-bound}, we combine these two stepping-up
lemmas with the graph lower bound and monotonicity in the uniformity to
prove the lower bound in Theorem~\ref{thm:main}(ii).

\begin{theorem}\label{thm:stepup-low-uniformity}
    Fix $p \geq 1$ and let $C = \binom{p+1}{2}$. Then for all integers $q > p C$, $Q\ge (p+1)q^C$ and $n \geq 2$, we have
    $$
        A_{p+2}(3pn;Q,p)>2^{A_2(n;q, pC)-1}.
    $$
\end{theorem}

\begin{theorem}\label{thm:main-stepup}
    For an integer $p \geq 2$, there exists a constant $C_p$ depending only on $p$ such that the following holds for all $k \geq 2p+2$ and $s = \lceil (k-1)/p \rceil$. For all integers $q > 2p$, $Q \geq q^{C_p}$, and $n \geq 10k$, it holds that 
    $$
        A_k(3pn; Q, p) \geq 2^{M},
    $$
    where $M= \min\Bigl\{A_{k-1}(n; Q, p), A_{h}(n, q, t): t\in [p-1],(s-2)t+2\le h\le k-1 \Bigr\}-1$.
\end{theorem}

\begin{remark}\label{rem:p=1-main-stepup}
The monochromatic case $p=1$ is also covered by the same stepping-up
argument, in the following stronger form. For every $q\ge 2$, every
$k\ge 4$, and every $n\ge 10k$,
\[
    A_k(3n;q,1)\ge 2^{A_{k-1}(n;q,1)-1}.
\]
Indeed, when $p=1$ the family indexed by $t\in[p-1]$ is empty, so the
quantity $M$ in Theorem~\ref{thm:main-stepup} becomes simply
$A_{k-1}(n;q,1)-1$. In the proof there are no admissible pairs, no
auxiliary colorings $\phi_{h,t}$, and no matrix $M(V)$. Thus the
all-phases part of the argument does not appear. The construction reduces
to the usual monochromatic stepping-up coloring: monotone
$\delta$-sequences are colored by the lower-uniformity coloring, while
non-monotone $\delta$-sequences are handled by the standard first-local-
extremum rule, using two fixed colors from the $q$-palette. Hence no
enlarged palette $Q$ is needed.
We state Theorem~\ref{thm:main-stepup} only for $p\ge2$ because this is precisely
the range in which the all-phases mechanism and the auxiliary colorings indexed by $t\in [p-1]$ are needed.
\end{remark}

\subsection{Proof of the endpoint stepping-up lemma}\label{subsec:first-endpoint}

\begin{proof}[Proof of Theorem~\ref{thm:stepup-low-uniformity}]
Let $M=A_2(n;q,pC)-1$ and $U=\{0,1\}^M$. To prove the theorem, it suffices to construct a $(p+1)q^C$-coloring of $U^{(p+2)}$ which does not contain $P_{3pn}^{(p+2)}$ using at most $p$ colors. By the definition of $A_2(n;q,pC)$, there is a $q$-coloring $\phi$ of $[M]^{(2)}$ with no copy of $P_n^{(2)}$ using at most $pC$ colors. Let $[q]$ be the palette of $\phi$. 

For any $V\in U^{(p+2)}$, since $\Delta(V)=(\delta_1,\delta_2,\dots,\delta_{p+1})$ has the unique maximum property, there is a unique index $i(V)\in [p+1]$ such that $\delta_{i(V)}$ is the maximum of $\Delta(V)$. To construct our desired coloring, we now define a $(p+1)\times(p+1)$ matrix $M(V)=(m_{ij})$ in the following way.

\begin{enumerate}[label=\textup{(\roman*)}]
    \item For $i, j\in[p+1]$ such that either $i=j$ or $\delta_i=\delta_j$, the $(i,j)$ entry $m_{ij}$ of $M(V)$ is $1$,

    \item for $i\ne j\in[p+1]$ with $\delta_i\ne \delta_j$, the $(i,j)$ entry $m_{ij}$ of $M(V)$ is $ \phi(\delta_i,\delta_j)$.
\end{enumerate}

Observe that $M(V)$ is a symmetric matrix whose diagonal entries are $1$. Since $\phi$ is a $q$-coloring, the number of possible matrices of the form $M(V)$ is at most $q^C$.

We now define a coloring $\chi$ of $U^{(p+2)}$ as $\chi(V)=(i(V),M(V))$. Then, $\chi$ is a $(p+1)q^C$-coloring of $U^{(p+2)}$.

Suppose, for contradiction, that $\chi$ contains a copy of $P_{3pn}^{(p+2)}$ using at most $p$ colors. Let $v_1<v_2<\dots<v_{3pn+p+1}$ be the vertices of such a path and $S=(v_1,v_2,\dots,v_{3pn+p+1})$. Also, let $\Delta=\Delta(S)=\{\delta_i\}_{i=1}^{3pn+p}$ and $V_i=\{v_i,v_{i+1},\dots,v_{i+p+1}\}$. Then we have $\Delta(V_i)=(\delta_i,\delta_{i+1},\dots,\delta_{i+p})$.

Since the $(p+2)$-uniform  monotone path on $S$ uses at most $p$ colors, by the definition of $\chi$, there is $t\in[p+1]$ such that $i(V)\ne t$ for every $V\in E_S$ where $E_S=\{V_i:i\in [3pn]\}$. Indeed, if there are $V_{j_1},V_{j_2},\dots,V_{j_{p+1}}$ such that $i(V_{j_{t}})=t$ for all $t\in [p+1]$, then the first coordinates of $\chi(V_{j_1}),\chi(V_{j_2}),\dots,\chi(V_{j_{p+1}})$ are pairwise different, i.e., $|\chi(E_S)|>p$. This is a contradiction because the $(p+2)$-uniform monotone path on $S$ uses at most $p$ colors. Hence there exists $t\in[p+1]$ such that $i(V)\ne t$ for every $V\in E_S$.

Let $I=[t,3pn-1+t]$. Thus $i\in I$ implies that there is $j\in [3pn]$ such that $i=j+t-1$.

We now note that the interval $I$ has size $3pn$, it
contains an index $i_0$ whose distance from both ends of $I$ is at least $pn$. We will construct a sequence of indices $i_0,i_1,\ldots,i_n$, maintaining the property that $i_s$ has distance at least $p(n-s)$ from both ends of $I$. This property holds for $s=0$ by the choice of $i_0$. Suppose $i_s$ is chosen for some $s<n$. Since $i_s\in I$, there is $j_s\in [3pn]$ such that $i_s=j_s+t-1$, i.e., $\delta_{i_s}$ is the $t$-th term of $\Delta(V_{j_s})$, hence it is not a maximum in $\Delta(V_{j_s})$. Therefore, there exists a unique index $j_s\le i_{s+1}\le j_s+p$ such that $\delta_{i_{s+1}}$ is the maximum of $\Delta(V_{j_s})$. Then we have $\delta_{i_{s+1}}>\delta_{i_s}$ and $|i_{s+1}-i_s|\le p$. Hence $i_{s+1}$ has distance at least $p(n-s)-p=p(n-s-1)$ from both ends of $I$, so the property holds. Thus the construction continues for $n$ steps.

Consequently, the inequality $\delta_{i_0}<\delta_{i_1}<\cdots<\delta_{i_{n}}\in [M]$ holds and consecutive indices differ by at most $p$. These $n+1$ increasing $\delta$-values form a copy of $P_n^{(2)}$ on $[M]^{(2)}$.

It remains to bound the number of colors used by this path under $\phi$. 
For each $s\in[n-1]\cup \{0\}$, the $t$-th term of $\Delta(V_{j_s})$ is $\delta_{i_s}$. Let $u_s=i_{s+1}-j_s+1$ be the position of $\delta_{i_{s+1}}$ in $\Delta(V_{j_s})$. Then the $(t,u_s)$ entry of $M(V_{j_s})$ is $\phi(\delta_{i_s},\delta_{i_{s+1}})$. Since $\chi(V_{i_0}),\dots, \chi(V_{i_n})$ uses at most $p$ colors, there are at most $p$ pairwise different matrices among $M(V_{j_0}),\dots, M(V_{j_{n-1}})$. Therefore, the number of element of $[q]$ that appears as entries of $M(V_{j_s})$ for some $s$ is at most $p\binom{p+1}{2}=pC$. In other words,
$$
|\{\phi(\delta_{i_s},\delta_{i_{s+1}}):s\in \{0,1,\dotsm,n-1\}\}|\le pC.
$$
This is a contradiction because $\phi$ does not induce a copy of $P_n^{(2)}$ using at most $pC$ colors.
\end{proof}


\subsection{Proof of the main stepping-up lemma}
\label{subsec:large-range}

\begin{proof}[Proof of Theorem~\ref{thm:main-stepup}]
Let $r=2p+1$, $Q\ge q^{(p^3+r)}=q^{(p^3+2p+1)}$, and $U=\{0,1\}^M$. Since $k\ge 2p+2$, we have $r\le k-1$. By the definition of $M$, there is a $Q$-coloring $\Phi$ of $[M]^{(k-1)}$ with no copy of $P_n^{(k-1)}$ using at most $p$ colors. Define a pair $(h,t)$ of positive integers to be \emph{admissible}, if $2\le h\le k-1$, $1\le t\le p-1$, and $(s-2)t+2\le h$. Then, by the definition of $M$, for every admissible pair $(h,t)$ there exists a $q$-coloring $\phi_{h,t}$ of $[M]^{(h)}$ with no $P_n^{(h)}$ using at most $t$ colors.

For every $\alpha\in\{1,\ldots,p\}$, $\beta\in\{1,\ldots,\alpha\}$, and $V\in U^{(k)}$ with $\Delta(V)=(\delta_1,\dots,\delta_{k-1})$, let $\Delta_{\alpha,\beta}(V)=(\delta_\beta,\delta_{\beta+\alpha},\dots,\delta_{\beta+(h_{\alpha,\beta}-1)\alpha})$ where $h_{\alpha,\beta}$ is the maximum integer such that $\beta+(h_{\alpha,\beta}-1)\alpha\le k-1$. For any $t\le p-1$, we say $V$ is \emph{$(\alpha,\beta,t)$-admissible} if all terms of $\Delta_{\alpha,\beta}(V)$ are pairwise different and $(h_{\alpha,\beta},t)$ is admissible. Observe that $h_{\alpha,\beta}$ does not depend on $V$. We now define a $3$-dimensional $(p\times p\times p)$ matrix $M(V)=(m_{ij\ell})_{i,j,\ell\le p}$ whose entries all lie in $[q]$ in the following way.

\begin{enumerate}[label=\textup{(\roman*)}]
    \item For $j\le i\le p$ and $\ell\le p-1$, if $V$ is $(i,j,\ell)$-admissible, then $m_{ij\ell}\defeq\phi_{h_{i,j},\ell}(\Delta_{i,j}(V))$, 
    
    \item If either $i< j\le p$ or $V$ is not $(i,j,\ell)$-admissible, then $m_{ij\ell}\defeq1$.
\end{enumerate}

Again, let $V\in U^{(k)}$ with $\Delta(V)=(\delta_1,\dots,\delta_{k-1})$. Define $P(V)$ as the pattern of $\{\delta_1,\delta_2,\dots,\delta_{r}\}$. Since the number of $r$-patterns is at most $r^r$, we consider $P(V)$ as an element of $[r^r]$. 

To prove the theorem, it suffices to construct a $Q$-coloring $\chi$ of $U^{(k)}$ which does not contain $P_{3pn}^{(k)}$ using at most $p$ colors. First, we assume that the palette of $\Phi$ contains $[r^r]\times [q]^{p\times p \times p}$ where $[q]^{p\times p \times p}$ is the set of all $(p\times p \times p)$ matrices whose entries are in $[q]$. Since $Q\ge q^{(p^3+r)}\ge r^r\times q^{(p^3)}$, the assumption is valid. We now define $\chi$ in the following way.

    \begin{enumerate}[label=(R\arabic*), ref=(R\arabic*)]
        \item \label{it:R1} If $\Delta(V)$ is monotone, then $\chi(V) \defeq \Phi(\Delta(V))$. Otherwise,
    
        \item \label{it:R2} $\chi(V)\defeq (P(V),M(V))$.
    \end{enumerate}

To complete the proof, we must show that $\chi$ does not contain  $P_{3pn}^{(k)}$ using at most $p$ colors. Suppose, for contradiction, that $\chi$ contains $P_{3pn}^{(k)}$ using at most $p$ colors, i.e., there is a strictly increasing sequence $S=(v_1,v_2,\dots v_{3pn+k-1})$ of $U$ which induces a $k$-uniform  monotone path with at most $p$ colors under $\chi$.

Let $\Delta=\Delta(S)=(\delta_1,\delta_2,\dots,\delta_{3pn+k-2})$ and $V_i=\{v_i,v_{i+1},\dots,v_{i+k-1}\}$ for all $i\in [3pn]$. We note that $\Delta(V_i)=\{\delta_i,\delta_{i+1},\dots,\delta_{i+k-2}\}$.

\begin{claim}\label{clm:no-monotone-interval}
The sequence $\Delta$ is not strictly monotone on any interval of length $n+k-2$.    
\end{claim}

\begin{claimproof}[Proof of Claim~\ref{clm:no-monotone-interval}]
By the definition of $\chi$, if there is such an interval $J=\{a,a+1,\dots,a+(n+k-3)\}\subseteq [3pn+k-1]$, then $\Delta(V_i)$ is strictly monotone for all $i\in \{a,a+1,\dots,a+(n-1)\}$. By \ref{it:R1}, we have
$$
    \chi(V_i)=\Phi(\Delta(V_i))
$$
for all $i\in \{a,a+1,\dots,a+(n-1)\}$.

Therefore, the $(k-1)$-uniform  monotone path induced by a strictly increasing sequence $(\delta_a,\delta_{a+1},\dots,\delta_{a+(n+k-3)})$ uses at most $p$ colors under $\Phi$ because the edges of the path are precisely $\Delta(V_a),\dots,\Delta(V_{a+n-1})$. This is a contradiction.    
\end{claimproof}

\begin{claim}\label{clm:no-monotone}
    For all $i\in [(3p-1)n-k]$, the sequence $\Delta(V_i)$ is non-monotone.
\end{claim}

\begin{claimproof}[Proof of Claim~\ref{clm:no-monotone}]
    Suppose there is an index $i\in [(3p-1)n-k]$ such that $\Delta(V_i)$ is monotone. Without loss of generality, we assume that $\Delta(V_i)$ is strictly increasing. 
    Let $j \geq i+k-1$ be the smallest integer such that $(\delta_{i},\delta_{i+1},\ldots,\delta_{j})$ is non-monotone. Since $\Delta$ is not monotone in any interval of length $n+k-2$ by Claim~\ref{clm:no-monotone-interval}, such a $j$ exist and 
    $$
    i+(k-1) \le j \leq 3pn-1.
    $$
    By the choice of $j$, the sequence $\Delta(V_{j  - \ell})$ is non-monotone for every $2\le\ell\le p+2$. Indeed, since $k\ge 2p+2\ge p+4$, we have
    $$
    j-\ell<j-1<j\le j-\ell+(k-2).
    $$
    Therefore, both $\delta_{j-\ell}$ and $\delta_{j}$ are captured by $\Delta(V_{j-\ell})$ for all $2\le\ell\le p+2$. Since $\delta_{j-\ell}<\delta_{j-\ell+1}$ and $\delta_{j-1}>\delta_{j}$, we deduce that $\Delta(V_{j-\ell})$ is non-monotone for all $2\le\ell\le p+2$.
    
    Furthermore, the $r$-patterns of $\Delta(V_{j-2}),\Delta(V_{j-3}),\dots,\Delta(V_{j-(p+2)})$ are pairwise different because, by the minimality of $j$, the sequence $(\delta_i,\delta_{i+1},\dots,\delta_{j-1})$ is strictly increasing, while $\delta_{j-1}>\delta_{j}$. Hence, in $\Delta({V_{j-\ell}})$, the first local maximum occurs at the position $\ell$, and these positions are pairwise distinct. Thus, the $k$-uniform monotone path induced by $S$ spans at least $(p+1)$ colors under $\chi$. This is a contradiction.
    \end{claimproof}

Let $\Delta'=(\delta_1,\delta_2,\dots,\delta_{n'})$ where $n'=(3p-1)n-k+(r-1)$. Since $p\ge2$, $n\ge 10k$, $k\ge2p+2$, and $r=2p+1$, we have $n'\ge 2r+5p$. Furthermore, $\Delta'$ has the unique maximum property and satisfies $\calP(\Delta';r)\le p$ because $\Delta$ does. Therefore, by Corollary~\ref{cor:monotone-subsequence}, the $(2p,2p)$-reduction $\Delta_{2p}'$ satisfies the properties in Corollary~\ref{cor:monotone-subsequence}. Then, the $(p+2)$-pattern sequence of $\Delta_{2p}'$ is periodic with least period $d$ for some $d\le p$. Furthermore, the least period $d$ should be strictly larger than $1$, because if $d=1$, by \ref{it:P3}, the reduction $\Delta'_{2p}$ is strictly increasing. This contradicts to Claim~\ref{clm:no-monotone}. 

For convenience, denote $\Delta_{2p}'$ as $\Delta^*=(\eta_1,\eta_2,\dots,\eta_{n'-4p})$ and for $a\in [d]$, denote $\Delta_{a}^*$ for the subsequence of $\Delta^*$ consisting of all terms $\eta_i$ such that $1\le i\le n'-4p$ and $i \equiv a \pmod d$.
Observe that for any $i\in [n'-4p]$, we have $\eta_i=\delta_{i+2p}$ in $\Delta$. By \ref{it:P3} of Corollary~\ref{cor:monotone-subsequence}, there is a $\sigma\in[d]$ such that $\Delta_{\sigma}^*$ is strictly monotone. Without loss of generality, we assume that $\Delta_{\sigma}^*$ is strictly increasing. If the subsequence is strictly decreasing, we apply the same argument to the reversed sequence.

For each $a\in [d]$ and $j$ such that $a+d(j-1)+(k-2)\le n'-4p$, let $h_a$ be the number of indices $j'$ such that $\sigma+d(j'-1)\in[a+d(j-1),a+d(j-1)+(k-2)]$. 
We note that $h_a\ge2$, because $k-1\ge2p+1$ and $h_a$ is independent of $j$, since shifting $j$ by one shifts the interval by $d$ and does not change the number of integers congruent to $\sigma$ modulo $d$ inside it. Moreover,
$$
h_a=\left|\left\{t\in\{0,\ldots,k-2\}:a+t\equiv \sigma \pmod d\right\}\right|.
$$
Thus, when we sum over all $a\in[d]$, each offset
$t\in\{0,\ldots,k-2\}$ is counted exactly once, namely for the unique
$a\in[d]$ with $a+t\equiv\sigma\pmod d$. Consequently,
$$
\sum_{a\in[d]}h_a=k-1.
$$

We are now ready to complete the proof by showing the following claim.

\begin{claim}\label{clm:color-count}
The $k$-uniform  monotone path induced by $S$ spans at least $p+1$ colors under $\chi$.
\end{claim}

\begin{claimproof}[Proof of Claim~\ref{clm:color-count}]
For each $a\in [d]$, put $C_a=\lfloor(h_a-2)/(s-2)\rfloor$. We first show that $\sum_a C_a\ge p-d+1$.  Suppose not.  Then $\sum_a C_a\le p-d$.  Since $h_a-2\le (s-2)(C_a+1)-1$ for every $a$, summing over all the choices of $a$ gives
$$
(k-1)-2d\le (s-2)p-d.
$$
Thus $k-1\le (s-2)p+d\le(s-1)p$.  But $s-1=\lceil(k-1)/p\rceil-1$, so the definition of the ceiling gives $k-1>(s-1)p$, a contradiction.

Choose integers $t_a$ with $0\le t_a\le C_a$ and $\sum_a t_a=p-d+1$.  Since $d\ge2$, every $t_a$ is strictly smaller than $p$. Recall that $\eta_i=\delta_{i+2p}$ for all $i\in [n'-4p]$. Define
$$
E_a=\{V_{i}:2p+1\le i\le n'-2p-(k-1),(i-2p) \equiv a \pmod d\}.
$$
Then, for any $V_i\in E_a$, we have
$$
\Delta(V_i)=(\delta_i,\dots,\delta_{i+{k-2}})=(\eta_{a+d(j-1)},\dots,\eta_{a+d(j-1)+(k-2)}).
$$
for some $j$. Recall that $n'=(3p-1)n-k+(r-1)$, Hence $n'-2p-(r-1)\le (3p-1)n-k$. Therefore, for every $V_i\in E_a$, the $\delta$-sequence $\Delta(V_i)$ is not monotone by Claim~\ref{clm:no-monotone}. This implies that $\chi(V_i)=(P(V_i),M(V_i))$ for all $V \in E_a$ by~\ref{it:R2}. We next show that
$$
|\chi(E_a)|\ge t_a+1.
$$
If $t_a=0$, this is trivial because the family is nonempty. Suppose $t_a>0$ and $|\chi(E_a)|\le t_a$. Since $t_a\le C_a$, the pair $(h_a,t_a)$ is admissible. Indeed,
\begin{equation}\label{eq:(h_a,t_a)admissible}
    (s-2)t_a+2\le (s-2)C_a+2\le h_a. 
\end{equation}
Let $\beta= \sigma-a+1\bmod{d}$. Then, the equality $h_{d,\beta}=h_a$ holds. Therefore, the inequality (\ref{eq:(h_a,t_a)admissible}) guarantees that every $V\in E_a$ is $(d,\beta,t_a)$-admissible, because $\Delta_{d,\beta}(V)$ is pairwise different since $\Delta^*_{\sigma}$ is strictly increasing. Hence $|\chi(E_a)|\le t_a$ implies that $\{\Delta_{d,\beta}(V_i):V_i\in E_a\}$ spans at most $t_a$ colors under $\phi_{h_a,t_a}$ where
$$
\Delta_{d,\beta}(V_{i})=(\delta_{i+(\beta-1)},\delta_{i+(\beta-1)+d},\dots,\delta_{i+(\beta-1)+d(h_a-1)}).
$$
Since $n\ge 10k$ and $k\ge 2p+2$, we have
$$
\sigma+2p+d(n+h_a-1)\le 3p+dn+k \le(3p-1)n-k+(r-1)-4p=n'-4p,
$$
the sequence
$$
S_\sigma=(\delta_{\sigma+2p+d},\delta_{\sigma+2p+2d},\dots, \delta_{\sigma+2p+d(n+h_a-1)})
$$
is strictly monotone. Furthermore, the set $\{\Delta_{d,\beta}(V_i):V_i\in E_a\}$ contains all consecutive $h_a$-block of $S_\sigma$. Since $\{\Delta_{d,\beta}(V_i):V_i\in E_a\}$ spans at most $t_a$ colors under $\phi_{h_a,t_a}$, the $h_a$-uniform  monotone path induced by $S_\sigma$ uses at most $t_a$ colors under $\phi_{h_a,t_a}$. This is a contradiction because $\phi_{h_a,t_a}$ does not contain a copy of $P_n^{(h_a)}$ using at most $t_a$ colors. Therefore, the inequality $|\chi(E_a)|\ge t_a+1$ holds.

Finally, by \ref{it:P2} of Corollary~\ref{cor:monotone-subsequence}, for $a_1\neq a_2\in [d]$ and any $V\in E_{a_1}$ and $W\in E_{a_2}$, the $r$ patterns $P(V)$ and $P(W)$ are different. Since $\chi(V)=(P(V),M(V))$ and $\chi(W)=(P(W),M(W))$, we have $\chi(V)\neq \chi(W)$. Therefore, for $E_S=\{V_i:i\in [3pn-(k-1)]\}$, we have 
$$
|\chi(E_S)|\ge |\chi(E_1)|+|\chi(E_2)|+\dots+|\chi(E_d)|\ge \sum_{a\in[d]}(t_a+1)=d+\sum_{a\in[d]}t_a\ge p+1.
$$
Therefore, the $k$-uniform  monotone path induced by $S$ spans at least $p+1$ colors under $\chi$.
\end{claimproof}

Claim~\ref{clm:color-count} contradicts the assumption that the $k$-uniform  monotone path induced by $S$ uses at most $p$ colors. Hence the constructed $Q$-coloring of $[2^M]^{(k)}$ contains no monotone $P_{3pn}^{(k)}$ using at most $p$ colors. This gives $A_k(3pn;Q,p)>2^M$, as claimed.
\end{proof}


\section{Proof of the lower bound}\label{sec:lower-bound}

We now prove Theorem~\ref{thm:main}(ii). The proof is an induction on
the allowed color complexity $p$, with a secondary induction on the
uniformity $k$. The two stepping-up lemmas from
Section~\ref{sec:step-up} provide the recursive step; before beginning
the induction, we record two elementary ingredients that serve as the
graph-level input and allow us to pass lower bounds to larger
uniformities.

The first ingredient is a graph-level lower bound. It can also be
deduced from a result of Fox, Sudakov, and Wigderson~\cite{Fox-Sudakov-Wigderson}
on color-avoiding directed paths in transitive tournaments; for
completeness, we include the short direct proof.

\begin{lemma}\label{lem:A_2}
    For positive integers $n,q,p$ with $q > p$, we have $A_2(n;q,p) \geq  (n^{1/p}-1)^{q}$.
\end{lemma}

\begin{proof}[Proof of Lemma~\ref{lem:A_2}]
If $n=1$, the statement is trivial.  So
assume $n\ge2$. Set $m=\lceil n^{1/p}\rceil-1$. 
Let $G$ be the complete ordered graph on $[m]^q$, ordered lexicographically:
for distinct vertices $x=(x_1,\ldots,x_q)$ and $y=(y_1,\ldots,y_q)$, write $x<y$ if
$x_i<y_i$ at the first coordinate $i$ where $x_i\ne y_i$.  Color the edge
$xy$, with $x<y$, by this first differing coordinate $i$.

We claim that this coloring has no monotone path of length $n$ using at most
$p$ colors.  Suppose such a path is
$x^{(0)}<x^{(1)}<\cdots<x^{(n)}$, and suppose all its edge colors lie in
$S\subseteq[q]$, where $|S|\le p$.
For each vertex $x$, let $x|_S$ be its projection to the coordinates in $S$.
Along every edge of the path, the first coordinate where the endpoints differ
belongs to $S$.  Hence the projected tuples
$x^{(0)}|_S,x^{(1)}|_S,\ldots,x^{(n)}|_S$ strictly increase in lexicographic
order.
But there are only $m^{|S|}\le m^p<n$ possible projected tuples.  This is
impossible for a strictly increasing sequence of $n+1$ projected tuples.
Therefore the coloring avoids every monotone path of length $n$ using at most
$p$ colors.

The graph has $m^q$ vertices, so $A_2(n;q,p)>m^q$.  Since
$m\ge n^{1/p}-1$, we get the desired estimate $A_2(n;q,p)\ge (n^{1/p}-1)^q$.
\end{proof}

The second ingredient is monotonicity in the uniformity: any lower bound
for smaller uniformity automatically gives the same lower bound for all
larger uniformities.

\begin{lemma}\label{lem:monotone-in-uniformity}
For any integers $r>s\ge 2$, $q,p\ge 1$, and $n\in\mathbb N$,
\[
    A_r(n;q,p)\ge A_s(n;q,p).
\]
\end{lemma}

\begin{proof}[Proof of Lemma~\ref{lem:monotone-in-uniformity}]

It is enough to prove the case $r=s+1$.  Fix $k\ge 2$, and set
$M=A_k(n;q,p)-1$.  By the definition of $A_k(n;q,p)$, there is a $q$-coloring
$\phi$ of $[M]^{(k)}$ with no copy of $P_n^{(k)}$ using at most $p$ colors.

Define a $q$-coloring $\Phi$ of $[M]^{(k+1)}$ as follows.  If
$v_1<\cdots<v_{k+1}$, set
$\Phi(\{v_1,\ldots,v_{k+1}\})=\phi(\{v_1,\ldots,v_k\})$.
Suppose, for contradiction, that there is a copy of $P_n^{(k+1)}$
using at most $p$ colors, with vertices $v_1<\cdots<v_{n+k}$.  For
$i=1,\ldots,n$, the color of the $i$th edge is
$\phi(\{v_i,\ldots,v_{i+k-1}\})$.  Hence
$v_1<\cdots<v_{n+k-1}$ form a copy of $P_n^{(k)}$ under $\phi$ using at most
$p$ colors, a contradiction.
Thus $A_{k+1}(n;q,p)>M=A_k(n;q,p)-1$, so
$A_{k+1}(n;q,p)\ge A_k(n;q,p)$.  Iterating gives the result.
 \end{proof}

We are now ready to prove the lower bound in Theorem~\ref{thm:main}(ii).

\begin{proof}[Proof of Theorem~\ref{thm:main}(ii)]

For any positive integer $m$, let $g_m(x) = \lfloor x^{\frac{1}{(2m^3)^m}} \rfloor$ and $f_m(x)=x^{(2m^3)^m}$. Observe that the function $g_m$ is an increasing function that tends to infinity for each $m$. 
We proceed by induction on $p$ to show that 
\begin{equation}\label{eq:goal}
    A_k(n;q,p) \geq T_{s}\left(\left(\frac{n}{(4p)^k}\right)^{g_{p}(q)/2p^3}\right)
\end{equation}
holds for all $k \ge 2$, $q \ge f_p(2p^3)$, and $n \geq N(k,q,p) = C_0(4p)^{qk}$, where $s=\lceil (k-1)/p \rceil$ and $C_0$ is a sufficiently large absolute constant. 

For the base case $p=1$, the case $k=2$ follows from
Lemma~\ref{lem:A_2}. The case $k=3$ follows from
Theorem~\ref{thm:stepup-low-uniformity} with $p=1$, applied with
$\lfloor q/2\rfloor$ colors in the graph construction. For $k\ge4$,
we repeatedly apply Remark~\ref{rem:p=1-main-stepup}. This gives the
required tower-height $k-1$ lower bound.

Now we may assume that $p \geq 2$ and the statement holds for all $p' < p$. To prove the statement for $p$, we use a secondary induction on $k$.
\begin{case}\label{case:k=<p+1}
$k\le p+1$.
\end{case}

By Lemma~\ref{lem:monotone-in-uniformity} and Lemma~\ref{lem:A_2}, we have 
$$
A_k(n;q,p)\ge A_2(n;q,p)\ge(n^{1/p}-1)^q\ge \left(\frac{n}{(4p)^k}\right)^{g_{p}(q)/(2p^3)}.
$$
This proves the base cases of the secondary induction on $k$.

\begin{case}\label{case:k=<2p+1}
$p+2\le k\le 2p+1$.    
\end{case}
By Proposition~\ref{lem:monotone-in-uniformity}, it suffices to show that
$$
    A_{p+2}(n; q, p) \geq T_2\left(\left(\frac{n}{(4p)^k}\right)^{g_{p}(q)/(2p^3)}\right).
$$

Let $C=\binom{p+1}{2}$.  Since $q\ge f_p(2p^3)$, we have
$g_p(q)\ge 2p^3\ge pC+1$.  Also, from $g_p(q)\le q^{1/(2p^3)}$, we get
$q\ge g_p(q)^{2p^3}\ge (p+1)g_p(q)^C$.
Therefore, by applying Proposition~\ref{lem:monotone-in-uniformity} and Theorem~\ref{thm:stepup-low-uniformity}, we have
\begin{equation}\label{eq:(k=<p+2)1}
A_{k}(n; q, p)\ge A_{p+2}(n; q, p) \geq T_2(A_2(\lfloor n/3p \rfloor;g_{p}(q),pC)-1).    
\end{equation}
Furthermore, Observation~\ref{lem:A_2} guarantees that 
\begin{equation}\label{eq:(k=<p+2)2}
A_2(\lfloor n/3p \rfloor;g_{p}(q),pC)-1\ge (\lfloor n/3p \rfloor^{1/(pC)}-1)^{g_{p}(q)}-1\ge  \left(\frac{n}{(4p)^k}\right) ^{g_{p}(q)/(2p^3)}.    
\end{equation}
Combining (\ref{eq:(k=<p+2)1}) and (\ref{eq:(k=<p+2)2}) yields
$$
A_{k}(n; q, p)\geq T_2(A_2(\lfloor n/3p \rfloor;g_{p}(q),pC)-1)\ge  T_2\left(\left(\frac{n}{(4p)^k}\right) ^{g_{p}(q)/(2p^3)}\right).
$$
This completes the proof of Case~\ref{case:k=<2p+1}.

\begin{case}{\label{case:k=>2p+2}} 
    $k\ge 2p+2$.
\end{case}
We now assume $k \geq 2p+2$, and suppose the equation (\ref{eq:goal}) holds for all $k' < k$, for any $q \geq f_p(2p^3)$ and $n \geq N(k', q, p)$. Since $q$ is sufficiently large in terms of $p$, we have
$\lfloor q^{1/(2p^3)}\rfloor>p$ and
$\lfloor q^{1/(2p^3)}\rfloor^{p^3}\le q$.  Hence
Theorem~\ref{thm:main-stepup} implies $A_k(n;q,p)\ge 2^M$, where
\[
M+1=
\min\left\{
    A_{k-1}\!\left(\left\lfloor \frac{n}{3p}\right\rfloor;q,p\right),\
    A_h\!\left(\left\lfloor \frac{n}{3p}\right\rfloor;
    \left\lfloor q^{1/(2p^3)}\right\rfloor,t\right)
\right\}.
\]
Here the minimum is taken over the first term and over all terms of the second
form with $t\in[p-1]$ and $(s-2)t+2\le h\le k-1$.

\begin{claim}\label{clm:M+1}
    $M\ge T_{s-1}\left(\left(\frac{n}{(4p)^{k}}\right)^{g_{p}(q)/(2p^3)}\right)$.
\end{claim}

\begin{claimproof}[Proof of Claim~\ref{clm:M+1}]
It is enough to show that every term in the minimum defining $M+1$ is at least
the right-hand side of the claim plus $1$.

First consider the term
$A_{k-1}(\lfloor n/(3p)\rfloor;q,p)$.  By the induction hypothesis on $k$ and
the inequality $s-1\le \lceil (k-2)/p\rceil$, we have
\[
\begin{aligned}
A_{k-1}\left(\left\lfloor \frac{n}{3p}\right\rfloor;q,p\right)
&\ge
T_{\lceil (k-2)/p\rceil}\left(
\left(
\frac{\lfloor n/(3p)\rfloor}{(4p)^{k-1}}
\right)^{g_p(q)/(2p^3)}
\right)  \\
&\ge
T_{s-1}\left(\left(\frac{n}{(4p)^k}\right)^{g_p(q)/(2p^3)}\right)+1.
\end{aligned}
\]

Now consider a term of the second form, with $t\in[p-1]$ and
$(s-2)t+2\le h\le k-1$.  By the induction hypothesis on $p$, and since
$s-1\le \lceil (h-1)/t\rceil$, we get
\[
\begin{aligned}
A_h\left(\left\lfloor \frac{n}{3p}\right\rfloor;
\left\lfloor q^{1/(2p^3)}\right\rfloor,t\right)
&\ge
T_{\lceil (h-1)/t\rceil}\left(
\left(
\frac{\lfloor n/(3p)\rfloor}{(4p)^h}
\right)^{
g_t(\lfloor q^{1/(2p^3)}\rfloor)/(2p^3)}
\right)  \\
&\ge
T_{s-1}\left(\left(\frac{n}{(4p)^k}\right)^{g_p(q)/(2p^3)}\right)+1.
\end{aligned}
\]
The last inequality follows because $h\le k-1$,
$\lfloor n/(3p)\rfloor/(4p)^h\ge n/(4p)^k$ for sufficiently large $n$, and
$g_t(\lfloor q^{1/(2p^3)}\rfloor)\ge g_p(q)$ for every $t<p$, once $q$ is
sufficiently large in terms of $p$.
\end{claimproof}

Since $A_k(n;q,p)>2^M$, the claim gives
$A_k(n;q,p)>
    T_s\left(\left(\frac{n}{(4p)^k}\right)^{g_p(q)/(2p^3)}\right)$, as required.
\end{proof}

\section{Concluding remarks}

We have determined the tower height of $A_k(n;q,p)$ for every fixed $p$
once $q$ is sufficiently large in terms of $p$. It remains natural to
remove this largeness assumption on $q$. We believe that the same tower
height is correct for every nontrivial number of colors $q>p$.

\begin{conjecture}\label{conj:small-q}
For every integer $p\ge 1$, there exists a constant $c_p>0$, depending
only on $p$, such that for every $k\ge 2$ and all sufficiently large $n$,
\[
    A_k(n;p+1,p)
    \ge
    T_{\lceil (k-1)/p\rceil}\left(n^{c_p}\right).
\]
\end{conjecture}

Together with Theorem~\ref{thm:main}(i),
Conjecture~\ref{conj:small-q} would determine the tower height of
$A_k(n;q,p)$ for every $q>p$, namely $\lceil (k-1)/p\rceil$, since the
case $q=p+1$ implies the lower bound for all larger $q$ by monotonicity
in $q$.

We also mention a consequence for a question from the first arXiv
version of this paper. There, the first and second authors asked whether
for each fixed $p\ge 1$ there is an integer $a_p$ such that
\[
    A_{k+a_p}(n;q,p+1) \ge A_k(n;q,p)
\]
holds for all $k$ and all sufficiently large $n$ and $q$
\cite[Question~5.2]{Choi-Lee}. Our results give a negative answer. For
$q$ sufficiently large, the tower height of the left-hand side is
$\left\lceil (k+a_p-1)/(p+1)\right\rceil$,
whereas the tower height of the right-hand side is
$\left\lceil (k-1)/p\right\rceil$.
For every fixed $a_p$, the latter is larger when $k$ is sufficiently
large. Thus, in the monotone-path setting, allowing one additional color
on the desired path cannot be compensated for by any bounded increase in
the uniformity. This contrasts with the corresponding color-avoidance problems for
complete hypergraphs, where the uniformity parameter can have a much
stronger influence on the Ramsey numbers; see, for example, the results
of Dubroff, Gir{\~a}o, Hurley, and Yap~\cite{Dubroff-Girao-Hurley-Yap}.

\section*{Declaration on the use of generative AI}

The third author used GPT-5.5 Pro as a sounding board during the later stages of the development of this paper. The main ideas and proofs were
developed by the authors, with the exception of the short block-compression argument used for the upper bound in Theorem~\ref{thm:main}(i), which was suggested by GPT-5.5 Pro. All
arguments were checked and finalized by the authors, who take full responsibility for the content of the paper.

\section*{Acknowledgements}
The first and second authors thank Eion Mulrenin and Eero R\"{a}ty for helpful comments. They also thank their advisor, Jaehoon Kim, for his continued support and encouragement.


\vspace{-0.2cm}

\providecommand{\MR}[1]{}
\providecommand{\MRhref}[2]{%
  \href{http://www.ams.org/mathscinet-getitem?mr=#1}{#2}
}

    \bibliographystyle{amsplain_initials_nobysame}
    \bibliography{bibfile}

@article {Mulrenin-Pohoata-Zakharov,
    AUTHOR = {Mulrenin, Eion and Pohoata, Cosmin and Zakharov, Dmitrii},
     TITLE = {Color avoidance for monotone paths},
   JOURNAL = {Discrete Anal.},
  FJOURNAL = {Discrete Analysis},
      YEAR = {2025},
     PAGES = {Paper No. 23, 14},
      ISSN = {2397-3129},
   MRCLASS = {05C55 (05D10)},
  MRNUMBER = {4975150},
}

@article {MS,
    AUTHOR = {Moshkovitz, Guy and Shapira, Asaf},
     TITLE = {Ramsey theory, integer partitions and a new proof of the {E}rd{\H{o}}s-{S}zekeres theorem},
   JOURNAL = {Adv. Math.},
  FJOURNAL = {Advances in Mathematics},
    VOLUME = {262},
      YEAR = {2014},
     PAGES = {1107--1129},
      ISSN = {0001-8708,1090-2082},
   MRCLASS = {05D10},
  MRNUMBER = {3228450},
MRREVIEWER = {Peter\ D.\ Johnson, Jr.},
       DOI = {10.1016/j.aim.2014.06.008},
       URL = {https://doi.org/10.1016/j.aim.2014.06.008},
}

@article {Dubroff-Girao-Hurley-Yap,
    AUTHOR = {Dubroff, Quentin and Gir{\~a}o, Ant{\'o}nio and Hurley, Eoin and
              Yap, Corrine},
     TITLE = {Tower gaps in multicolour {R}amsey numbers},
   JOURNAL = {Forum Math. Sigma},
  FJOURNAL = {Forum of Mathematics. Sigma},
    VOLUME = {11},
      YEAR = {2023},
     PAGES = {Paper No. e84, 15},
      ISSN = {2050-5094},
   MRCLASS = {05C55},
  MRNUMBER = {4651621},
MRREVIEWER = {Tomasz\ Dzido},
       DOI = {10.1017/fms.2023.89},
       URL = {https://doi-org.libra.kaist.ac.kr/10.1017/fms.2023.89},
}

@article {StepUp,
    AUTHOR = {Erd{\H{o}}s, P. and Hajnal, A. and Rado, R.},
     TITLE = {Partition relations for cardinal numbers},
   JOURNAL = {Acta Math. Acad. Sci. Hungar.},
  FJOURNAL = {Acta Mathematica. Academiae Scientiarum Hungaricae},
    VOLUME = {16},
      YEAR = {1965},
     PAGES = {93--196},
      ISSN = {0001-5954,1588-2632},
   MRCLASS = {04.60},
  MRNUMBER = {202613},
MRREVIEWER = {L.\ Gillman},
       DOI = {10.1007/BF01886396},
       URL = {https://doi-org.libra.kaist.ac.kr/10.1007/BF01886396},
}

@article {ErdosSzekeres,
    AUTHOR = {Erd{\H{o}}s, P. and Szekeres, G.},
     TITLE = {A combinatorial problem in geometry},
   JOURNAL = {Compositio Math.},
  FJOURNAL = {Compositio Mathematica},
    VOLUME = {2},
      YEAR = {1935},
     PAGES = {463--470},
      ISSN = {0010-437X,1570-5846},
   MRCLASS = {99-04},
  MRNUMBER = {1556929},
       URL = {http://www.numdam.org/item?id=CM_1935__2__463_0},
}

@article {MorseHedlund,
    AUTHOR = {Morse, Marston and Hedlund, Gustav A.},
     TITLE = {Symbolic {D}ynamics},
   JOURNAL = {Amer. J. Math.},
  FJOURNAL = {American Journal of Mathematics},
    VOLUME = {60},
      YEAR = {1938},
    NUMBER = {4},
     PAGES = {815--866},
      ISSN = {0002-9327,1080-6377},
   MRCLASS = {99-04},
  MRNUMBER = {1507944},
       DOI = {10.2307/2371264},
       URL = {https://doi-org.libra.kaist.ac.kr/10.2307/2371264},
}

@article {Fox-Pach-Sudakov-Suk,
    AUTHOR = {Fox, Jacob and Pach, J\'anos and Sudakov, Benny and Suk,
              Andrew},
     TITLE = {Erd{\H o}s-{S}zekeres-type theorems for monotone paths and
              convex bodies},
   JOURNAL = {Proc. Lond. Math. Soc. (3)},
  FJOURNAL = {Proceedings of the London Mathematical Society. Third Series},
    VOLUME = {105},
      YEAR = {2012},
    NUMBER = {5},
     PAGES = {953--982},
      ISSN = {0024-6115,1460-244X},
   MRCLASS = {05D10 (05C65 52C10 91A43)},
  MRNUMBER = {2997043},
MRREVIEWER = {David\ Conlon},
       DOI = {10.1112/plms/pds018},
       URL = {https://doi-org.libra.kaist.ac.kr/10.1112/plms/pds018},
}

@article {Pohoata-Zakharov,
    AUTHOR = {Pohoata, Cosmin and Zakharov, Dmitrii},
     TITLE = {On the number of high-dimensional partitions},
   JOURNAL = {Proc. Lond. Math. Soc. (3)},
  FJOURNAL = {Proceedings of the London Mathematical Society. Third Series},
    VOLUME = {128},
      YEAR = {2024},
    NUMBER = {2},
     PAGES = {Paper No. e12586, 12},
      ISSN = {0024-6115,1460-244X},
   MRCLASS = {05A17 (06A11)},
  MRNUMBER = {4753776},
MRREVIEWER = {Mikl\'os\ B\'ona},
       DOI = {10.1112/plms.12586},
       URL = {https://doi-org.libra.kaist.ac.kr/10.1112/plms.12586},
}

@article {Falgas-Ravry-Raty-Tomon,
    AUTHOR = {Falgas-Ravry, Victor and R{\"a}ty, Eero and Tomon, Istv\'an},
     TITLE = {Dedekind's problem in the hypergrid},
   JOURNAL = {Adv. Math.},
  FJOURNAL = {Advances in Mathematics},
    VOLUME = {488},
      YEAR = {2026},
     PAGES = {Paper No. 110796, 41},
      ISSN = {0001-8708,1090-2082},
   MRCLASS = {05A16 (05A17 05D10)},
  MRNUMBER = {5018899},
       DOI = {10.1016/j.aim.2026.110796},
       URL = {https://doi-org.libra.kaist.ac.kr/10.1016/j.aim.2026.110796},
}

@article {Erdos-Szemeredi,
    AUTHOR = {Erd{\H o}s, P. and Szemer{\'e}di, A.},
     TITLE = {On a {R}amsey type theorem},
   JOURNAL = {Period. Math. Hungar.},
  FJOURNAL = {Periodica Mathematica Hungarica. Journal of the J\'anos Bolyai
              Mathematical Society},
    VOLUME = {2},
      YEAR = {1972},
     PAGES = {295--299},
      ISSN = {0031-5303,1588-2829},
   MRCLASS = {05C30},
  MRNUMBER = {325446},
MRREVIEWER = {J.\ E.\ Graver},
       DOI = {10.1007/BF02018669},
       URL = {https://doi-org.libra.kaist.ac.kr/10.1007/BF02018669},
}

@article {set-ramsey-lower,
    AUTHOR = {Arag{\~a}o, Lucas and Collares, Maur{\'i}cio and Marciano, Jo{\~a}o
              Pedro and Martins, Ta{\'i}sa and Morris, Robert},
     TITLE = {A lower bound for set-coloring {R}amsey numbers},
   JOURNAL = {Random Structures Algorithms},
  FJOURNAL = {Random Structures \& Algorithms},
    VOLUME = {64},
      YEAR = {2024},
    NUMBER = {2},
     PAGES = {157--169},
      ISSN = {1042-9832,1098-2418},
   MRCLASS = {05C55 (05C80 05D40)},
  MRNUMBER = {4704267},
MRREVIEWER = {Yuval\ Wigderson},
       DOI = {10.1002/rsa.21173},
       URL = {https://doi-org.libra.kaist.ac.kr/10.1002/rsa.21173},
}

@article {set-ramsey-code,
    AUTHOR = {Conlon, David and Fox, Jacob and He, Xiaoyu and Mubayi, Dhruv
              and Suk, Andrew and Verstra\"ete, Jacques},
     TITLE = {Set-coloring {R}amsey numbers via codes},
   JOURNAL = {Studia Sci. Math. Hungar.},
  FJOURNAL = {Studia Scientiarum Mathematicarum Hungarica. Combinatorics,
              Geometry and Topology (CoGeTo)},
    VOLUME = {61},
      YEAR = {2024},
    NUMBER = {1},
     PAGES = {1--15},
      ISSN = {0081-6906,1588-2896},
   MRCLASS = {05C55 (94B99)},
  MRNUMBER = {4752601},
MRREVIEWER = {Yuval\ Wigderson},
}

@article{Seidenberg,
    AUTHOR = {Seidenberg, A.},
     TITLE = {A simple proof of a theorem of {E}rd{\H o}s and {S}zekeres},
   JOURNAL = {J. London Math. Soc.},
  FJOURNAL = {Journal of the London Mathematical Society},
      YEAR = {1959},
    VOLUME = {34},
     PAGES = {352--353},
   MRCLASS = {05.00},
  MRNUMBER = {0101288},
}

@unpublished{Fox-Sudakov-Wigderson,
  title={Color-avoiding directed paths in tournaments},
  author={Fox, Jacob and Sudakov, Benny and Wigderson, Yuval},
  note={\path{arXiv:2512.10438}},
  year={2025}
}

@article{Loh,
  title={Directed paths: from {R}amsey to {R}uzsa and {S}zemer{\'e}di},
  author={Loh, Po-Shen},
  note={\path{arXiv:1505.07312}},
  year={2015}
}

@article {Gowers-Long,
    AUTHOR = {Gowers, W. T. and Long, J.},
     TITLE = {The length of an {$s$}-increasing sequence of {$r$}-tuples},
   JOURNAL = {Combin. Probab. Comput.},
  FJOURNAL = {Combinatorics, Probability and Computing},
    VOLUME = {30},
      YEAR = {2021},
    NUMBER = {5},
     PAGES = {686--721},
      ISSN = {0963-5483,1469-2163},
   MRCLASS = {05D99 (05C35)},
  MRNUMBER = {4298506},
       DOI = {10.1017/s0963548320000371},
       URL = {https://doi-org.libra.kaist.ac.kr/10.1017/s0963548320000371},
}

@book {UMP,
    AUTHOR = {Graham, Ronald L. and Rothschild, Bruce L. and Spencer, Joel
              H.},
     TITLE = {Ramsey theory},
    SERIES = {Wiley Series in Discrete Mathematics and Optimization},
   EDITION = {Paperback},
 PUBLISHER = {John Wiley \& Sons, Inc., Hoboken, NJ},
      YEAR = {2013},
     PAGES = {xiv+196},
      ISBN = {978-1-118-79966-6},
   MRCLASS = {05-02 (05A99 05C55 54H20)},
  MRNUMBER = {3288500},
}

@book{LindMarcus,
    AUTHOR = {Lind, Douglas and Marcus, Brian},
     TITLE = {An introduction to symbolic dynamics and coding},
 PUBLISHER = {Cambridge University Press},
   ADDRESS = {Cambridge},
      YEAR = {1995},
}

@incollection{Ruzsa-Szemeredi,
    AUTHOR = {Ruzsa, I. Z. and Szemer{\'e}di, E.},
     TITLE = {Triple systems with no six points carrying three triangles},
 BOOKTITLE = {Combinatorics},
    SERIES = {Colloq. Math. Soc. J{\'a}nos Bolyai},
    VOLUME = {18},
     PAGES = {939--945},
 PUBLISHER = {North-Holland},
   ADDRESS = {Amsterdam-New York},
      YEAR = {1978},
}

@incollection{Steele,
    AUTHOR = {Steele, J. Michael},
     TITLE = {Variations on the monotone subsequence theme of {E}rd{\H{o}}s and {S}zekeres},
 BOOKTITLE = {Discrete Probability and Algorithms},
    SERIES = {IMA Vol. Math. Appl.},
    VOLUME = {72},
     PAGES = {111--131},
 PUBLISHER = {Springer},
   ADDRESS = {New York},
      YEAR = {1995},
}

@unpublished{Choi-Lee,
  title={A stepping-up lemma for monotone paths with bounded color complexity},
  author={Choi, Jigang and Lee, Hyunwoo},
  note={\path{arXiv:2605.12318}},
  year={2026}
}


\end{document}